\documentclass[12pt]{amsart}
\usepackage{cite}
\usepackage{amsmath}
\usepackage{amssymb,amsfonts,amscd,verbatim,hyperref,cleveref,graphics}
\usepackage{xspace,xcolor}

\newtheorem{thm}{Theorem}[section]
\newtheorem{prop}[thm]{Proposition}
\newtheorem{cor}[thm]{Corollary}
\newtheorem{lemma}[thm]{Lemma}
\numberwithin{equation}{section}

\theoremstyle{definition}
\newtheorem{rem}[thm]{Remark}
\newtheorem{ex}[thm]{Example}
\newtheorem{quest}[thm]{Question}
\newtheorem{prob}[thm]{Problem}

\def\RR{{\mathbb R}}
\def\QQ{{\mathbb Q}}

\DeclareSymbolFont{bbold}{U}{bbold}{m}{n}
\DeclareSymbolFontAlphabet{\mathbbold}{bbold}
\def\one{\mathbbold{1}}
\def\zero{\mathbbold{0}}

\newcommand{\zs}

\newcommand{\term}[1]{{\textit{\textbf{#1}}}}   

\newcommand{\taun}{\tau_{\mathrm{un}}}

\begin{document}

\title{On separability of unbounded norm topology}

\author{M.~Kandi\' c}
\address{Faculty of Mathematics and Physics,
University of Ljubljana,
Jadranska 19,
SI-1000 Ljubljana,
Slovenija}
\address{Institute of Mathematics, Physics and Mechanics,
Jadranska 19,
SI-1000 Ljubljana,
Slovenija}
\email{marko.kandic@fmf.uni-lj.si}

\author{A.~Vavpeti\v c}
\address{Faculty of Mathematics and Physics,
University of Ljubljana,
Jadranska 19,
SI-1000 Ljubljana,
Slovenija}
\address{Institute of Mathematics, Physics and Mechanics,
Jadranska 19,
SI-1000 Ljubljana,
Slovenija}
\email{ales.vavpetic@fmf.uni-lj.si}


\thanks{
The first author acknowledges financial support from the Slovenian Research Agency, Grants No. P1-0222, J1-2453 and J1-2454. The second author acknowledges financial support from the Slovenian Research Agency, Grants No. P1-0292, J1-8131, N1-0064 and N1-0083.
}

\begin{abstract}
In this paper, we continue the investigation of topological properties of unbounded  norm (un-)topology in normed lattices. We characterize separability and second countability of un-topology in terms of properties of the underlying normed lattice. We apply our results to prove that an order continuous Banach function space  $X$ over a semi-finite measure space is separable if and only if it has a $\sigma$-finite carrier and  is separable with respect to the topology of local convergence in measure. We also address the question when a normed lattice is a normal space with respect to the un-topology. 
\end{abstract}

\maketitle

\section{Introduction}

Recently, the study of unbounded convergences in vector and normed lattices attracted many researchers in the area. Unbounded order convergence, although already introduced by Nakano \cite{Nakano:48} and studied by DeMarr \cite{DeMarr:64} and Kaplan \cite{Kaplan:97}, its systematic study was initiated by Gao, Troitsky and Xanthos (see  \cite{GaoX:14, Gao:14, GTX}). While unbounded norm (un-) convergence and topology were formally introduced in \cite{DOT}, the study of topological properties of un-topology began in \cite{KMT}. In \cite{KLT}, the authors extended un-topology beyond normed lattice and studied its properties. Generalization of un-topology to locally solid vector lattices and other instances of unbounded convergences were studied in e.g. \cite{DEM18, KT18, Tay18, Tay19, Zab18}. In this paper we continue the investigation of topological properties of un-topology. In particular, we are interested in separablity, second countability, and normality.  

The paper is structured as follows. In \Cref{section: prelim} we introduce notatation and basic notions needed throughout the text. In \Cref{splosno separabilnost} we consider the question when a given normed lattice $X$ equipped with its unbounded norm topology  $\taun$ is a separable or even a second countable space. We first prove that $X$ contains a countable quasi-interior set $\{u_n:\; n\in \mathbb N\}$  whenever $(X,\taun)$ is separable. This result enables us to prove (see  \Cref{sep norm un}) that separability and/or second countability of $(X,\taun)$ is equivalent to separability of the normed lattice $X$ itself. This result is applied to prove that separability of the unbounded norm topology passes up and down between $X$ and its Banach space completion $\widehat X$. 

In \Cref{Section components separability} we make one step further. Instead of considering separability of $X$ or equivalently separability of principal ideals $I_{u_n}$ for each $n \in \mathbb N$, we
rather consider separability of the sets $C_{u_n}$ of components of vectors $u_n$ equipped with the metric topology inherited from $X$. For a normed lattice with the projection property we prove in \Cref{separability in terms of components} that it is separable if and only if the set of components $C_{u_n}$ is separable for each $n\in\mathbb N$. We also provide an example that this statement does not hold in general for normed lattices without the principal projection property. 

In \Cref{section: BFS} we consider separability of order continuous Banach function spaces over semi-finite measure spaces. In \Cref{separabilityBFSEXT} we prove that such a Banach function space is separable if and only if it has a $\sigma$-finite carrier and it is separable with respect to the topology of local convergence in measure. 

In \Cref{Section: Metric space components} we are interested in those vector lattices for which the equality $I_u=I_v$ between principal ideals  or the equality $B_u=B_v$
between principal bands yields an existence of a bijective mapping between $C_u$ and $C_v$. In \Cref{arch unif compl components}  we prove that the equality $I_u=I_v$ for positive vectors in a uniformly complete Archimedean vector lattice  guarantees an existence of a surjective isometry between $C_u$ and $C_v$. If we replace the condition $I_u=I_v$ by the weaker condition $B_u=B_v$, then \Cref{discontinuity of varphi} shows that it is even possible that there is no homeomorphism between $C_u$ and $C_v$ in the case of Dedekind complete vector lattices. Nevertheless,  for order continuous Banach lattices the condition $B_u=B_v$ is sufficient for $B_u$ and $B_v$ to be homeomorphic (see \Cref{homeomorphism between space of components}). 

In  \Cref{Section: normal un} we discuss the separation axioms of a normed lattice equipped with the unbounded norm topology. Since $(X,\taun)$ satisfies the Hausdorff axiom, by the general theory of topological vector spaces the space $(X,\taun)$ is completely regular. We prove that $(X,\taun)$ is a paracompact and hence a normal space whenever $X$ is an atomic KB-space.

\section{Preliminaries}\label{section: prelim}

Throughout the paper we will assume that all vector lattices are Archimedean.  Let $X$ be a vector lattice. The set of all positive elements is denoted by $X^+$. For a subset $A$ we denote by $I_A$ and $B_A$ the ideal and the band in $X$ generated by $A$, respectively. In the case when the set $A=\{u\}$ is a singleton set  we write $I_A=I_u$ and $B_A=B_u$.
Ideals and bands generated by a single element are called \term{principal}. 
For a given set $A\subseteq X$ we denote the set $\{x\in X:\; |x|\wedge |a|=0 \textrm{ for all }a\in A\}$ by $A^d$ and we call it the \term{disjoint complement} of $A$ in $X$. It turns out that $A^d$ is always a band in $X$.
If a band $B$ satisfies $X=B\oplus B^d$, then $B$ is called a \term{projection band}. To a projection band $B$ we associate the \term{band projection} $P_B$ which is the identity mapping on $B$ and the zero mapping on $B^d$. It is well known that band projections always commute. 
If every principal band in $X$ is a projection band, then $X$ is said to have the \term{principal projection property}.  A vector lattice $X$ is said to be \term{Dedekind complete} or \term{order complete} if non-empty bounded from above subsets of $X$ have suprema. If non-empty countable subsets of $X$ have suprema, then $X$ is said to be \term{$\sigma$-Dedekind complete} or \term{$\sigma$-order complete}. A vector lattice is said to satisfy the \term{countable sup property} whenever every non-empty subset possessing a supremum contains an at most countable subset possessing the same supremum. The countable sup property is equivalent to the following fact:  for each net $(x_\alpha)$ in $X$ that satisfies $0\leq x_\alpha\uparrow x$ there is an increasing sequence of indices $(\alpha_n)_{n\in\mathbb N}$ such that $0\leq x_{\alpha_n}\uparrow x$ (see e.g. \cite[Theorem 23.2(iii)]{Luxemburg:71}). A positive vector $a$ in $X$ is said to be an \term{atom} whenever the principal ideal $I_a$ is one-dimensional. Since $X$ is Archimedean, the principal ideal $I_a$ is a projection band in $X$. A vector lattice $X$ is called \term{atomic} if the linear span of the set of all atoms is \term{order dense} in $X$, i.e., for each non-zero positive vector $x\in X$ there exists an atom  $a$ in $X$ such that $0\leq a\leq x$.

A \term{lattice norm}  $\|\cdot\|$ on a vector lattice  is a norm  which satisfies $\|x\|\leq \|y\|$ for all elements $x$ and $y\in X$ with  $0\leq |x|\leq |y|$. A vector lattice equipped with a lattice norm is said to be a \term{normed lattice}. A \term{Banach lattice} is a normed lattice which is also a Banach space. The Banach space completion of a normed lattice $X$ which is denoted by $\widehat X$ is always a Banach lattice. 
A normed lattice is \term{order continuous} whenever $x_\alpha\downarrow 0$ implies $x_\alpha \to 0$ in norm.  A normed lattice $X$ has the \term{$\sigma$-Fatou property} if for every increasing sequence $(x_n)_{n\in\mathbb N}$ in $X^+$ with the supremum $x$ it follows that $$\|x\|=\sup_{n\in\mathbb N}\|x_n\|.$$ It turns out that dual Banach lattices and Banach lattice with order continuous norm have the $\sigma$-Fatou property (see e.g. \cite{MN91}). 

A set $\mathcal Q$ of positive vectors of a normed lattice $X$ is said to be a \term{quasi-interior set} if the ideal generated by $\mathcal Q$ in $X$ is dense in $X$. A positive vector $u\in X$ is said to be a \term{quasi-interior point} whenever the set $\{u\}$ is a quasi-interior set of $X$. A positive vector $u\in X$ is called a \term{weak unit} if $|x|\wedge u=0$ implies $x=0$. Since norm dense ideals are order dense, every quasi-interior point is a weak unit. Although the converse statement does not hold in general, in order continuous normed lattices  these two notions coincide. The remaining unexplained facts about vector and normed lattices can be found in \cite{Abramovich:02, Aliprantis:06, Luxemburg:71, MN91, SCH71, Zaanen:96}.

Given a normed lattice $X$, a positive number $\epsilon>0$ and a positive vector $x\in X^+$ we define the set  
$$U_{x,\epsilon}=\{y: \, \| |y|\wedge x\|<\epsilon\}.$$ 
Then  the collection of all sets of this form is a base of neighborhoods for zero for some locally solid Hausdorff linear topology on $X$. This topology is called the \term{unbounded norm topology} (un-topology for short) on $X$ induced by the norm, and it is denoted by $\taun$.  A net $(x_\alpha)$ in $X$ is said to \term{un-converge} to a vector $x\in X$ whenever for each $y\in X^+$ we have $|x_\alpha-x|\wedge y\to 0$ in norm. It should be clear that a net $(x_\alpha)$ un-converges to $x$ if and only if it converges to $x$ with respect to $\taun$.

\section{Separability of un-topology for general normed lattices}\label{splosno separabilnost}

In this section we consider the question under which conditions a normed lattice equipped with the unbounded norm topology $\taun$ is a separable or even a second countable topological space. Although in general the unbounded norm topology behaves differently than the norm topology, the answer is surprisingly simple. In fact, we prove in \Cref{sep norm un} that separability or equivalently second countability of $(X,\taun)$ is equivalent to separability of $X$. We start with the following proposition which at first glance might appear to be an improvement if \cite[Exercise 4.2.3]{Abramovich:02} however having in mind \Cref{sep norm un}, it just provides an invaluable tool for proving it.

\begin{prop}\label{splosni}
Let $X$ be a normed lattice such that $(X,\taun)$ is separable.
\begin{enumerate}
\item Then $X$ has a countable quasi-interior set.
\item If $X$ is a Banach lattice, then $X$ has a quasi-interior point.
\end{enumerate}
\end{prop}

\begin{proof}
(1) Let $\{x_n:\; n\in \mathbb N\}$ be a countable $\taun$-dense subset of $X$.
Since the inequality
$$0\leq |x^+-y^+|\wedge |z|\leq |x-y|\wedge |z|$$
holds for all $x,y,z \in X$, the set $\{x_n^+:\, n\in\mathbb N\}$ is $\taun$-dense in
$X^+$. Let $\mathcal F$ be the set of all finite subsets of $\mathbb N$. For $F\in \mathcal F$ we define
$x_F=\sum_{n\in F}x_n^+$ and we observe that the family $\mathcal D:=\{x_F:\; F\in \mathcal F\}$  is countable and upward directed.
The set  $\mathcal G$ of all finite sums of elements from $\mathcal D$ is countable, additive, upward directed and it obviously contains $\mathcal D$. Since $\mathcal G$ is upward directed, it can be considered as an increasing net.

Since $\mathcal G$ is additive, in order to prove that $\mathcal G$ is a quasi-interior set,  by \cite[Proposition 2.2]{Drnovsek01} it is enough to show that for every positive vector $x$  the net $(x-x\wedge u)_{u\in\mathcal G}$ converges to zero.
Pick any positive vector $x\in X^+$, any set $F_1\in \mathcal F$ and an arbitrary $\epsilon>0$. Then $x+U_{|x-x_{F_1}|,\epsilon}$ is a $\taun$-neighborhood of $x$, and since $\mathcal D$ is $\taun$-dense in $X^+$, there exists $F_2\in \mathcal F$ such that $x_{F_2}\in x+U_{|x-x_{F_1}|,\epsilon}$ from where it follows that
$$\big{\|}|x-x_{F_2}| \wedge |x-x_{F_1}|\big{\|}<\epsilon.$$
Let us define $F_0=F_1\cup F_2$. Then for any $u=\sum_{i=1}^n x_{G_i}\in \mathcal G$ with $F_0\subseteq G:=\bigcup_{i=1}^n G_i$ we have that
$x-u \leq x-x_{F_1}$ and $x-u\leq x-x_{F_2}$ from where it follows that
$(x-u)^+\leq (x-x_{F_1})^+$ and $(x-u)^+\leq (x-x_{F_2})^+$. Since each vector $a\in X$ satisfies $0\leq a^+\leq |a|$, we conclude that
\begin{align*}
\|(x-u)^+\|&=\|(x-u)^+ \wedge (x-u)^+\|\leq \|(x-x_{F_1})^+\wedge (x-x_{F_2})^+\|\\
&\leq \big{\|}|x-x_{F_1}| \wedge |x-x_{F_2}|\big{\|}<\epsilon.
\end{align*}
From this it follows that the net $((x-u)^+)_{u\in \mathcal G}$ is norm convergent to zero.
Due to the equality
$(x-u)^+=x-x\wedge u$, we conclude that for each $x\in X^+$ the net
$(x\wedge u)_{u\in \mathcal G}$ converges in norm to $x$. 

(2) By (1), there exists  a countable  quasi-interior set  $\mathcal D=\{u_n:\; n\in \mathbb N\}$ in $X$.
Since the series of positive vectors
$$\sum_{n=1}^\infty \frac{u_n}{2^n\,(\|u_n\|+1)}$$ converges absolutely, it converges to some positive vector $u\in X$.  Since for each $n\in \mathbb N$ we have $0\leq u_n\leq 2^n(\|u_n\|+1) u$, the principal ideal $I_u$ contains the quasi-interior set $\mathcal D$ from where it follows that $I_u$ is dense in $X$.
\end{proof}

The following theorem which follows from \Cref{splosni} is the main result of this section. 

\begin{thm}\label{sep norm un}
The following assertions about a normed lattice $X$ are equivalent.
\begin{enumerate}
\item $X$ is second countable.
\item $X$ is norm separable.
\item $(X,\taun)$ is separable.
\item $(X,\taun)$ is second countable.
\end{enumerate}
\end{thm}

\begin {proof}
The equivalence (1)$\Leftrightarrow$(2) holds for general metric spaces.

(2)$\Rightarrow$(3) If $X$ is norm separable, then $X$ is separable with respect to any topology weaker than the norm topology. In particular, $(X,\taun)$ is separable.

(3)$\Rightarrow$(4) Suppose that $(X,\taun)$ is separable. By \Cref{splosni} we conclude that  $X$ has a countable quasi-interior set, so that  $(X,\taun)$ is metrizable by \cite[Theorem 4.3]{KT18}. Since $(X,\taun)$ is metrizable and separable, it is second countable.

(4)$\Rightarrow$(2) Since $(X,\taun)$ is second countable, it is separable, so that by \Cref{splosni} we conclude that $X$ has a countable quasi-interior set $\mathcal G$. By enumerating the vectors in $\mathcal G$ as $\{u_1,u_2,\ldots\}$ and after replacing the $n$-th vector by the sum $u_1+\cdots+u_n$, we may assume without loss of generality that $u_1\leq u_2\leq \ldots$

Let us consider the order ideal $I_{\mathcal G}$ generated by $\mathcal G$. Since $\mathcal G$ is a quasi-interior set,  the ideal $I_{\mathcal G}$ is norm dense in $X$. 
We claim that the un-topology on $I_{\mathcal G}$ equals the relative topology induced by the un-topology from $X$. To see this, note first that the equality $U_{x,\epsilon}^{I_{\mathcal G}}:=\{y\in I_{\mathcal G} : \| |y|\wedge x\|<\epsilon \}=U_{x,\epsilon}\cap I_{\mathcal G}$ for all $x\in I_{\mathcal G}$ yields that every basis neighborhood $U_{x,\epsilon}^{I_{\mathcal G}}$ of zero for the un-topology on $I_{\mathcal G}$ is open in the relative topology. To prove the converse, pick any $x\in X$ and $\epsilon>0$. Then $U_{x,\epsilon}\cap I_{\mathcal G}$ is a basis open neighborhood for zero in the relative topology on $I_{\mathcal G}$ induced by the un-topology of $X$. Since $I_{\mathcal G}$ is norm dense in $X$, there exists a positive vector $y\in I_{\mathcal G}$ such that $\|x-y\|<\frac{\epsilon}{2}$. Pick any $z\in U_{y,\frac{\epsilon}{2}}^{I_{\mathcal G}}$. From the inequality
$$\big{\|}|z|\wedge x\big{\|}\leq \big{\|}|z|\wedge |x-y|\big{\|}+\big{\|}|z|\wedge y\big{\|}<\epsilon$$ we conclude that
$U_{y,\frac{\epsilon}{2}}^{I_{\mathcal G}}\subseteq U_{x,\epsilon}\cap I_{\mathcal G}$ which proves the claim.

Since $(X,\taun)$ is second countable, the un-topology on $I_{\mathcal G}$ is second countable as well. This yields that for each $n\in \mathbb N$ the interval $[-nu_n,nu_n]$ is second countable with respect to the relative un-topology induced from $I_{\mathcal G}$. Since on order intervals the un-topology agrees with the norm topology, each order interval $[-nu_n,nu_n]$ is norm separable. If $\mathcal F_n$ is a countable norm dense subset of $[-nu_n,nu_n]$, then the set $\mathcal F:=\bigcup_{n=1}^\infty \mathcal F_n$ is a countable norm dense subset of $I_{\mathcal G}$, so that norm density of $I_{\mathcal G}$ in $X$ yields norm density of $\mathcal F$ in $X$. This proves that $X$ is separable.
\end{proof}

Although in metric spaces separability is equivalent to second countability, this is not the case even for locally convex Hausdorff spaces. For example, the space $\mathbb R^{\mathbb R}$ equipped with the product topology is separable by \cite[Theorem 1]{RossStone}, yet the space is not second countable since it is an uncountable topological product. 

The following result tells that second countability and separability of the unbounded norm topology pass up and down between a normed lattice and its norm completion.

\begin{cor}
For a normed lattice $X$ the following assertions are equivalent. 
\begin{enumerate}
\item $(\widehat X,\taun)$ is second countable.
\item $(X,\taun)$ is second countable.
\item $(\widehat X,\taun)$ is separable.
\item $(X,\taun)$ is separable.
\end{enumerate}
\end{cor}

\begin{proof}
Equivalences (1)$\Leftrightarrow$(3) and (2)$\Leftrightarrow$(4) follow directly from \Cref{sep norm un}.

(1)$\Rightarrow$(2) By \Cref{sep norm un}, the Banach lattice $\widehat X$ is separable, so that $X$ is also separable. Another application of \Cref{sep norm un} yields that $(X,\taun)$ is separable. The converse implication can be proved similarly.
\end{proof}

The following corollary follows from \Cref{sep norm un} and from the fact that $X$ is separable whenever its norm dual $X^*$ is.

\begin{cor}
Let $X$ be a normed lattice. If $(X^*,\taun)$ is separable, then $(X,\taun)$ is separable.
\end{cor}

We conclude this section with a characterization of separability of atomic normed lattices with order continuous norm.

\begin{cor}\label{atomic un}
For an atomic normed lattice $X$ with order continuous norm the following statements are equivalent.
\begin{enumerate}
  \item $X$ is separable.
  \item $(X,\taun)$ is second countable.
  \item $(X,\taun)$ is separable.
  \item $(X,\taun)$ is first countable.
  \item $(X,\taun)$ is metrizable.
  \item $X$ has an at most countable quasi-interior set.
  \item Every disjoint set in $X$ is at most countable.
  \item Every disjoint set of atoms in $X$ is at most countable.
\end{enumerate}
\end{cor}

\begin{proof}
If $\dim X<\infty$, then all Hausdorff linear topologies on $X$ coincide so that, in this case, all statements between (1) and (5) are equivalent. Furthermore, $X$ has a strong unit and every disjoint set of atoms in $X$ has cardinality at most $\dim X<\infty$.

Assume now that $\dim X=\infty$.
Then the equivalences (1)$\Leftrightarrow$(2)$\Leftrightarrow$(3) follow from \Cref{sep norm un} while (5)$\Leftrightarrow$(6) follows from \cite[Theorem 4.3]{KT18}.
Since $\taun$ is a linear Hausdorff topology on $X$, (4)$\Leftrightarrow$(5) follows from \cite[p. 49]{Kelley:76}.
Furthermore, while \Cref{splosni} yields the implication (3)$\Rightarrow$(6), 
the implication (7)$\Rightarrow$(8) is obvious. 

(6)$\Rightarrow$(7) Suppose that $\mathcal D=\{u_n:\; n\in \mathbb N\}$ is an at most countable quasi-interior set in $X$ and pick a disjoint family  $\{x_\lambda:\;\lambda\in\Lambda\}$ of positive vectors in $X$. Then for each $n\in\mathbb N$ the family  $\{u_n\wedge x_\lambda:\;\lambda\in\Lambda\}$ is an order bounded disjoint family in $X$. For all $n,m\in \mathbb N$ we consider
the set
$$\Lambda_{n,m}:=\{\lambda\in \Lambda:\; \|u_n\wedge x_\lambda\|\geq \tfrac 1m\}.$$
Fix $n\in \mathbb N$ and suppose that $\Lambda_{n,m_0}$ is infinite for some $m_0\in \mathbb N$. Take any sequence $(y_k)_{k\in \mathbb N}$ in $\{u_n\wedge x_\lambda:\; \lambda \in \Lambda_{n,m_0}\}$ of distinct vectors. Since $X$ is order continuous, by \cite[Theorem 64.3]{LuxNotes2} the norm completion $\widehat X$ is also order continuous, so that the disjoint sequence $(y_k)_{k\in\mathbb N}$ converges in norm to zero by Theorem (see \cite[Theorem 4.14]{Aliprantis:06}). This contradicts the fact that $\|y_k\|\geq \frac{1}{m_0}$ for each $k\in \mathbb N$. Therefore, for all $n,m\in\mathbb N$ the set $\Lambda_{n,m}$ is at most countable.

Now consider the set $\Lambda_0:=\bigcup_{n,m=1}^\infty \Lambda_{n,m}$ which is clearly countable. If $\lambda\in \Lambda\setminus \Lambda_0$, then $x_\lambda\wedge u_n=0$ for each $n\in \mathbb N$. Since $\mathcal D$ is a quasi-interior set, we conclude that $x_\lambda=0$, from where it follows that $\Lambda$ is at most countable.

(8)$\Rightarrow$(1) Pick a maximal set $\mathcal A$ of atoms in $X$. Since $X$ is infinite-dimensional, we may enumerate the elements of $\mathcal A$ as $(e_n)_{n\in\mathbb N}$. For each $n\in \mathbb N$ consider the linear span $\mathcal J_n$ of the set $\{e_1,\ldots,e_n\}$. Since $\mathcal J_n$ is spanned by atoms, $\mathcal J_n$ is a projection band in $X$. Since each $\mathcal J_n$ is finite-dimensional, the union $F:=\bigcup_{n=1}^\infty \mathcal J_n$ is a separable subspace of $X$. Since $X$ is atomic, $F$ is order dense in $X$, and since the norm on $X$ is order continuous, $F$ is norm dense in $X$. This proves separability of $X$. 
\end{proof}

\begin{rem}
Some implications in \Cref{atomic un} do not hold if we relax certain assumptions. 

(1) The implications (6)$\Rightarrow$(7) and (8)$\Rightarrow$(1) in \Cref{atomic un} do not necessarily hold if we replace order continuity of the norm with Dedekind completeness. 
To see this, for a given index set $I$ consider the Dedekind complete Banach lattice $\ell^\infty(I)$ of all bounded sequences over the index set $I$. Clearly,  $\ell^\infty(I)$ contains a strong unit. 

If $I$ is uncountable, then $\ell^\infty(I)$ contains uncountably many pairwise disjoint atoms. 
On the other hand, in $\ell^\infty:=\ell^\infty(\mathbb N)$ every disjoint set of atoms is at most countable, yet  $\ell^\infty$ is not separable.

(2)  The implication (7)$\Rightarrow$(1) does not hold for order continuous Banach lattices without atoms. 
To see this, consider the circle $\mathbb S^1$ equipped with the Borel $\sigma$-algebra and the arc-length measure. Furthermore, consider the uncountable product of circles equipped with the product $\sigma$-algebra and the product probability measure $\mu$.  Since the coordinate functions are pairwise orthogonal, the Hilbert space $L^2(\mu)$ is not separable. Since the constant one function is a quasi-interior point in $L^2(\mu)$  similar arguments as in the proof of implication  (6)$\Rightarrow$(7) in \Cref{atomic un} show that every pairwise disjoint set in $L^2(\mu)$ is countable. 
\end{rem}

\section{Components of positive vectors and separability of un-topology}\label{Section components separability}

In \Cref{splosno separabilnost} we considered separability of un-topology in comparison with general topological properties of norm topology and un-topology. In this section we study separability of un-topology of normed lattices with the principal projection property in terms of separability of ``special" metric spaces which can be considered as their building blocks. These metric spaces are precisely spaces of components of positive vectors equipped with the metric induced by the norm of the underlying normed lattice. The motivation comes from Measure theory.

Given a finite measure space $(\Omega,\mathcal F,\mu)$ one can define the map $d_\mu\colon \mathcal F\times \mathcal F\to [0,\mu(\Omega)]$ as
$$d_\mu(A,B):=\|\chi_{A\triangle B}\|_1=\mu(A\setminus B)+\mu(B\setminus A).$$
It is easy to see that $d_\mu$ is a semi-metric on $\mathcal F$. On $\mathcal F$ we introduce the equivalence relation $\sim$ with
$$A\sim B \qquad \Leftrightarrow \qquad d_\mu(A,B)=0.$$
Then the semi-metric $d_\mu$ induces the well-defined metric $d$ on $\mathcal F/_\sim$ defined as
$$d([A],[B]):=d_\mu(A,B).$$ It should be clear that
the semi-metric space $(\mathcal F,d_\mu)$ is separable if and only if the metric space $(\mathcal F/_\sim,d)$ is separable.

The following proposition which immediately follows from Theorem 46.4 and Theorem 46.5 proved by Luxemburg in \cite{Lux14b} presents an important connection between separability of the norm topology (or equivalently, separability of the un-topology) and the local structure of normed lattices.

\begin{prop}\label{separability of components L1}
Suppose $(\Omega,\mathcal F,\mu)$ is a finite measure space.
Then $L^1(\mu)$ is separable if and only if the semi-metric space $(\mathcal F,d_\mu)$ is separable.
\end{prop}

Since each $L^1$-function is a limit of a sequence of step functions, the key role in the proof of \Cref{separability of components L1} actually play characteristic functions in $L^1(\mu)$ or equivalently their representing sets in $\mathcal F$. The concept of a characteristic function can be fruitfully generalized to the realm of vector lattices.
Given a positive vector $u$ of a vector lattice $X$, a positive vector $x$ is  a \term{component} of $u$ if $0\leq x\leq u$ and $x\perp (u-x)$. Since for each measurable set $A\subseteq X$ we have $\chi_A+\chi_{\Omega\setminus A}=\one$ and $\chi_A \wedge \chi_{\Omega\setminus A}=0$, the notion of a component of a positive vector naturally extends the notion of a characteristic function to vector lattices. The set of all components of a given vector $u$ is denoted by $C_u$. It should be clear that we always have $0,u\in C_u$. A component $x\in C_u$ is called \term{non-trivial} if $x$ is neither $0$ nor $u$. A linear combination of components of $u$ is called a \term{$u$-step function}. It is easy to see that we can always take pairwise disjoint components when working with a particular $u$-step function.  


%

%
%

The following result is clearly a generalization of \Cref{separability of components L1}. It will be used in \Cref{separabilityBFS} in order to obtain a characterization of separability of order continuous Banach function spaces over semi-finite measure spaces.

\begin{thm}\label{separability in terms of components}
For a normed lattice $X$ with the principal projection property the following statements are equivalent.
\begin{enumerate}
  \item $X$ is norm separable.
  \item $(X,\taun)$ is separable.
  \item $X$ contains a countable quasi-interior set and for each countable quasi-interior set $\mathcal Q$  and each positive vector $u\in \mathcal Q$ the metric space $(C_u,\rho)$ is separable.
  \item[(3')] $X$ contains a countable quasi-interior set $\mathcal Q$ such that for each positive vector $u\in \mathcal Q$ the metric space $(C_u,\rho)$ is separable.
\end{enumerate}
Moreover, if $X$ is a Banach lattice, then (1)-(3') are equivalent to each of the following statements.
\begin{enumerate}\setcounter{enumi}{3}
  \item $X$ contains a quasi-interior point and for each quasi-interior point $u$ the metric space $(C_u,\rho)$ is separable.
  \item[(4')] $X$ contains a quasi-interior point $u$ such that the metric space $(C_u,\rho)$ is separable.
\end{enumerate}
\end{thm}

\begin{proof}
While the equivalence (1)$\Leftrightarrow$(2) follows from \Cref{sep norm un}, implications (3)$\Rightarrow$(3') and (4)$\Rightarrow$(4') are obvious.

(2)$\Rightarrow$(3)
Suppose that $X$ is $\taun$-separable. By \Cref{splosni}, $X$ contains a countable quasi-interior set. Let $\mathcal Q$ be any quasi-interior set and pick any positive vector $u$ from $\mathcal Q$. Then $C_u$ equipped with the metric $d$ induced by the norm is a subspace of the metric space $(X,d)$. Since $X$ is separable metric space, its  subspace $C_u$ is separable as well.

(3')$\Rightarrow$(1)
Pick any positive vector $u\in X^+$ such that $C_u$ is separable, and let $\mathcal Q$ be a countable dense set of $C_u$. Let $\mathcal G$ be the set of all linear combinations of vectors from $\mathcal Q$ with rational coefficients.
We claim that $\mathcal G$ is dense in $I_u$. To this end, pick any positive vector $x\in I_u$ and $\epsilon>0$.
By Freudenthal's spectral theorem (see e.g. \cite[Theorem 33.2]{Zaanen:96}) there exists an increasing sequence $(s_n)_{n\in\mathbb N}$ of $u$-step functions such that $0\leq s_n\uparrow x$ and that $s_n\to x$ $u$-uniformly. In particular, $s_n\to x$ in norm, so that the positive part of the linear span of $C_u$ is dense in $I_u^+$. Due to the equality $x=x^+-x^-$ we immediately obtain that the linear span of $C_u$ is dense in $I_u$. Since $C_u$ is separable, by a standard approximation argument one can conclude that $I_u$ is separable as well.

For the general case, let us enumerate the vectors from $\mathcal Q$ as $\{u_1,u_2,\ldots\}$. For each $n\in \mathbb N$ we denote by $w_n$ the vector $u_1+\cdots+u_n$. From the identity $I_{w_n}=I_{u_1}+\cdots+I_{u_n}$ and the fact that the deal $I_{u_k}$ is separable for each $k\in\mathbb N$, we conclude that the ideal $I_{w_n}$ is separable as well. From this we immediately derive that the increasing union $X_0:=\bigcup_{n=1}^\infty I_{w_n}$ is separable in $X$.
Since $X_0$ is norm dense in $X$, we obtain separability of $X$.

Suppose now that $X$ is also norm complete. For the moreover statement it is enough to see that (1) implies (4). This follows immediately from the fact that every separable Banach lattice has a quasi-interior point and from the implication (1)$\Rightarrow$(3).
\end{proof}

The following example shows that \Cref{separability in terms of components} does not hold if one does not assume that the underlying normed lattice has the principal projection property.

\begin{ex}
Let $K$ be a product of uncountably many copies of the interval $[0,1]$. Since $K$ is connected, we have $C_{\one}=\{\zero,\one\}$, so that $C_{\one}$ is separable.
On the other hand, $C(K)$ is not separable as $K$ is not metrizable (see \cite[Theorem 4.1.3]{AK06}).
\end{ex}

\section{Applications to Banach function spaces}\label{section: BFS}

In this section we characterize separable Banach function spaces over semi-finite measure spaces in terms of measure-theoretical properties of the underlying measure space.
Given a measure space $(\Omega,\mathcal F,\mu)$, a \term{function space} $X$ is an (order) ideal in $L^0(\mu)$. A function space $X$ is a \term{Banach function space} if it is equipped with a complete lattice norm. The following lemma (see \cite[Theorem 46.2]{Lux14b}) shows that only order continuous Banach function spaces can be separable. 

\begin{lemma}
Every separable Banach function space in $L^0(\mu)$ is order continuous. 
\end{lemma}

Let $X$ be an order dense order continuous Banach function space in $L^0(\mu)$ where $\mu$ is finite. Norm separability of $X$ is equivalent to separability of $X$ equipped with un-topology. An application of \cite[Theorem 5.2]{KLT}
yields that $X$ is separable if and only if $X$ is separable with respect to the topology $\tau_\mu$ of convergence in measure.  In fact, separability of $(X,\taun)$ depends only on $L^0(\mu)$. 

\begin{lemma}\label{separabilnost gostost X L}
Let $(\Omega,\mathcal F,\mu)$ be a finite measure space and let $X$ be an order dense order continuous Banach function space in $L^0(\mu)$.
Then $X$ is norm separable if and only if $L^0(\mu)$ equipped with the topology of convergence in measure is separable.
\end{lemma}

\begin{proof}
Suppose first that $L^0(\mu)$ is separable. Since $\mu$ is finite and $X$ is order continuous, by \cite[Theorem 5.2]{KLT} the topology of convergence in measure is the un-topology on $L^0(\mu)$ induced by $X$.

Pick a weak unit $g$ in $X$. Due to order density of $X$ in $L^0(\mu)$ we have that $g$ is also a weak order unit in $L^0(\mu)$. Since the norm of $X$ is order continuous, $g$ is a quasi-interior point in $X$. Hence, by \cite[Theorem 3.3]{KLT} and its proof we conclude that the topology of convergence in measure on $L^0(\mu)$ is metrizable, and that the associated metric is given by
$$d_g(f_1,f_2)=\||f_1-f_2|\wedge g\|_X.$$

Since the un-topology on $X$ is just the restriction of un-topology induced by $X$ on $L^0(\mu)$, the un-topology on $X$ is precisely the topology of convergence in measure restricted to $X$. Since $(X,\tau_\mu|_X)$ is separable as a subspace of a separable metric space by \Cref{sep norm un} we conclude that $X$ is norm separable.

To prove the converse statement, assume that $X$ is norm separable. By \Cref{sep norm un} it follows that $X$ is $\taun$-separable. To finish the proof it suffices to prove that $X$ is $\tau_\mu$-dense in $L^0(\mu)$.
Pick any non-negative function $f\in L^0(\mu)$. Since $X$ is order dense in $L^0(\mu)$ and $L^0(\mu)$ has the countable sup property, there is an increasing sequence $(f_n)_{n\in \mathbb N}$ in $X$ such that $f_n\uparrow f$. This yields that $f_n\to f$ $\mu$-almost everywhere, and since
the measure $\mu$ is finite, by a consequence of the well-known Egorov's theorem we conclude that $f_n\to f$ in measure.
\end{proof}

\begin{prop}\label{finite measure separability X}
Let $(\Omega,\mathcal F,\mu)$ be a finite measure space. For an order continuous order dense Banach function space $X$ in $L^0(\mu)$ the following assertions are equivalent.
\begin{enumerate}
  \item $X$ is norm separable.
  \item $X$ is separable with respect to the topology of convergence in measure.
  \item The semi-metric space $(\mathcal F,d_\mu)$ is separable.
  \item The metric space $(C_f,\rho)$ is separable for each strictly positive function $f\in X$.
\end{enumerate}
If $\mathcal F$ is countably generated, then all the statements (1)-(4) hold.
\end{prop}

\begin{proof}
Equivalences between (1), (2) and (4) follow from \Cref{separability in terms of components}.

(3)$\Leftrightarrow$(1) If $(\mathcal F,d_\mu)$ is separable, then $(C_\one,\rho)$ is separable in $L^1(\mu)$. The implication (4)$\Rightarrow$(1) in the case of the Banach lattice $L^1(\mu)$ yields that $L^1(\mu)$ is norm separable, and so by an application of \Cref{separabilnost gostost X L} we conclude that $X$ is norm separable. The proof of the converse statement can be proved similarly and is therefore omitted.

Although the proof that the additional assertion implies (3) is well known, we will provide its sketch solely for the sake of completeness.
Suppose that a countable set $\mathcal F_0$ generates $\mathcal F$. Let $\mathcal F_1$ be the family of all sets which are obtained after finitely many steps of effecting Boolean operations on the sets of $\mathcal F_0$.  Since $\mathcal F_0$ is countable, $\mathcal F_1$ is countable as well. Let $\mathcal F_2$ be the closure of $\mathcal F_1$ in $\mathcal F$ with respect to the semi-metric $d_\mu$. The sets from $\mathcal F_2$ are precisely the sets which are approximable by sets from $\mathcal F_1$. Since the measure $\mu$ is finite, the complement of an approximable set is again an approximable set. Also, countable unions of approximable sets are also approximable. Hence, $\mathcal F_2$ is a $\sigma$-subalgebra in $\mathcal F$. Since $\mathcal F_2$ contains $\mathcal F_0$, we conclude $\mathcal F_2=\mathcal F$, and so $\mathcal F$ is separable as it is the closure of the countable set $\mathcal F_1$.
\end{proof}

In \Cref{separabilityBFS} we will extend \Cref{finite measure separability X} to order continuous Banach function spaces over semi-finite measures. First we need a general result about semi-finite measures. 

\begin{lemma}\label{semifinite non sigmafinite}
Let $(\Omega,\mathcal F,\mu)$ be a semi-finite measure space which is not $\sigma$-finite. Then there exists an uncountable disjoint family of measurable sets in $\Omega$ of finite positive measure.
\end{lemma}

\begin{proof}
Let $\mathcal A\subseteq \mathcal P(\Omega)$ be the family of all families of measurable pairwise disjoint subsets of $\Omega$ with finite positive measure.  Since $\mu$ is semi-finite, the family $\mathcal A$ is non-empty. The family $\mathcal A$ ordered by set inclusion is a partially ordered set. If $\mathcal C$ is a chain in $\mathcal A$, then it is obvious that the union of $\mathcal C$ is again in $\mathcal A$. Hence, by Zorn's lemma there exists a maximal element $\mathcal M$ in $\mathcal A$.

Suppose that $\mathcal M$ is countable. Since $E:=\bigcup_{F\in \mathcal M}F$ is a measurable subset of $\Omega$ and since $\mu$ is not $\sigma$-finite, the set $\Omega\setminus E$ has positive measure. Due to semi-finiteness of $\mu$ there exists a measurable set $F\subseteq \Omega\setminus E$ with a finite positive measure. This is in contradiction with maximality of $\mathcal M$.
\end{proof}

For a function space $X$ over a measure space $(\Omega,\mathcal F,\mu)$ and a measurable set $A\in\mathcal F$ we define the set $X_A=\{\chi_A f : f\in X\}$. Since $X$ is an ideal in $L^0(\mu)$, $X_A$ is an ideal in $X$, so that $X_A$ is a function space in $L^0(\mu)$. If $X$ is a Banach function space, then $X_A$ is a Banach function space as well. By $\mathcal F|_A$ we denote the relative $\sigma$-algebra on $A$, i.e., the $\sigma$-algebra of all measurable sets contained in $A$. Whenever $\mu(A)<\infty$, the restriction $d_{\mu|_A}$ of $d_\mu$ on $\mathcal F|_A$ induces a semi-metric. 

If $X$ is an order continuous Banach function space over a $\sigma$-finite measure space, then \cite[Theorem 5.2]{KLT} yields that $\taun$ agrees with the topology $\tau_\mu$ of local convergence in measure on sets of finite measure. An examination of its proof actually reveals that one can replace $\sigma$-finite measure spaces with semi-finite ones. Recall that a net $(f_\lambda)$ converges to $f$ in $L^0(\mu)$ with respect to the topology of local convergence in measure if $f_\lambda \chi_A \to f\chi_A$ in measure for each subset $A\in \mathcal F$ with finite measure.

\begin{lemma}\label{izboljsava 5.2}
Let $X$ be an order dense order continuous Banach function space over a semi-finite measure space. Then a net $(f_\alpha)$ in $X$ $\taun$-converges to zero if and only if it converges to zero with respect to the topology of local convergence in measure.
\end{lemma}

\begin{proof}
For the proof of the forward implication, pick a set $A$ of finite measure. Then $f_\alpha \wedge \chi_A\in X_A$, and since $f_\alpha\chi_A$ $\taun$-converges to zero in $X$, it also $\taun$-converges to zero in $X_A$. Since $\mu(A)<\infty$, \cite[Theorem 5.2]{KLT} yields that $(f_\alpha\chi_A)$ converges to zero in measure on $A$. This proves the forward implication. 

For the proof of the backward implication, assume that $(f_\alpha)$ converges to zero with respect to the topology of local convergence in measure. As in the proof of \cite[Theorem 5.2]{KLT} let $Z$ be the set of all functions in $X$ which vanish outside of a set of finite measure. We claim that $Z$ is a non-trivial ideal in $X$. Pick a non-zero function $f\in X^+$. Then there exists $n\in \mathbb N$ such that the set 
$A_n:=\{x\in \Omega:\; f(x)\geq \frac 1n\}$ has a positive measure.  Since $\mu$ is semi-finite, there exists a measurable set $A\subseteq A_n$ of finite positive measure.  Hence, $\chi_A \leq \chi_{A_n}\leq nf$ and since $X$ is an ideal in $L^0(\mu)$, we have that $\chi_A\in X$. This yields that $\chi_A\in Z$ and so $Z$ is non-trivial. The fact that $Z$ is an ideal in $X$ is clear. 

Now we claim that $|f_\alpha|\wedge g\to 0$ for each $g\in Z^+$. Since $g$ can be approximated by an increasing sequence of step functions and since $X$ is order continuous, it suffices to consider the case when $g$ itself is a step function. Furthermore, an application of the triangle inequality for the norm implies that we need to consider only the case when $g=\chi_A$ for some set $A$ of positive finite measure. Since $|f_\alpha|\wedge \chi_A\to 0$ in measure on $A$, by \cite[Theorem 5.2]{KLT} we have that $|f_\alpha|\wedge \chi_A\to 0$ in norm in $X_A$ and so in $X$. This proves that $|f_\alpha|\wedge g\to 0$ for each $g\in Z^+$. 

To conclude the proof, due to the standard approximation argument we need to show that $Z$ is norm dense in $X$. Since $X$ is order continuous, it suffices to prove that $Z$ is order dense in $X$. Pick a positive vector $f\in X^+$ and find a step function $s\in X^+$ such that $0<s\leq f$. Then there exists a measurable set $A$ and a scalar $\alpha>0$ such that 
$0<\alpha \chi_A\leq s\leq f$. Since $\mu$ is semi-finite, there is a subset $A'$ of $A$ of finite positive measure, from where it follows $0<\alpha \chi_{A'} \leq f$. To conclude the proof, note that we have $\chi_{A'}\in Z$.  
\end{proof}

The following two theorems are the main results of this section. 
They provide a characterization of separability of order continuous Banach function spaces.

\begin{thm}\label{separabilityBFS}
For an order continuous order dense Banach function space $X$ over a semi-finite measure space $(\Omega,\mathcal F,\mu)$ the following statements are equivalent.
\begin{enumerate}
  \item $X$ is norm separable.
  \item $X$ is separable with respect to the topology of local convergence in measure.
    \item The measure $\mu$ is $\sigma$-finite and the semi-metric space $(\mathcal F|_A,d_{\mu|_A})$ is separable for each set $A$ of finite measure.
  \item The measure $\mu$ is $\sigma$-finite and the function space $X_A$ is separable for each set $A$ of finite measure.
  \item $X$ admits strictly positive functions and the metric space $(C_f,\rho)$ is separable for any strictly positive function $f\in X$.
\end{enumerate}
\end{thm}

\begin{proof}
While the equivalence (1)$\Leftrightarrow$(2) follows from \Cref{sep norm un} and \Cref{izboljsava 5.2}, the equivalence (3)$\Leftrightarrow$(4) follows from \Cref{finite measure separability X}.

(1)$\Rightarrow$(4) Assume that $\mu$ is not $\sigma$-finite. By \Cref{semifinite non sigmafinite} there exists an uncountable family $\mathcal M$ of pairwise disjoint sets of finite positive measure.
Since $X$ is order dense in $L^0(\mu)$, for each $E\in \mathcal M$ there exists a function $g_E\in X$ such that $0<g_E\leq \chi_{E}$.
It is clear that that functions $g_E$ and $g_{E'}$ are disjoint whenever $E\cap E'=\emptyset$.
For each $n\in \mathbb N$ we introduce the family of functions in $X$ as
$$\mathcal F_n:=\{g_E:\; E\in \mathcal M \textrm{ and } \|g_E\|_X \geq \tfrac{1}{n}\}.$$
Since $\mathcal M$ is uncountable, there exists $m\in \mathbb N$ such that $\mathcal F_m$ is uncountable. For each $g_E\in \mathcal F_n$ consider the open ball $U_E$ with center in $g_E$ and radius $\frac{1}{2m}$. We claim that for different sets $E$ and $E'$ the balls $U_E$ and $U_{E'}$ are disjoint. Indeed, if $f$ is in their intersection, then $\|g_E-g_{E'}\|\leq \|g_E-f\|+\|f-g_{E'}\|<\frac{1}{m}$. However, this is in contradiction with 
$$\|g_{E}-g_{E'}\|=\||g_E-g_{E'}|\|=\|g_E+g_{E'}\|\geq \|g_E\|\geq \tfrac 1m.$$
Since $X$ admits an uncountable family of pairwise disjoint open sets, it cannot be separable. 

To prove the remaining claim of (4), pick a set $A$ of finite measure in $\Omega$ and note that separability of the normed space $X$ passes down to the Banach function space $X_A$. 

(4)$\Rightarrow$(1) Since $\mu$ is $\sigma$-finite, there exists an increasing sequence $(A_n)_{n\in\mathbb N}$ of sets of finite positive measure such that $\bigcup_{n=1}^\infty A_n=\Omega$. Since each space $X_{A_n}$ is separable by \Cref{finite measure separability X}, a simple topological argument yields separability of $Y:=\bigcup_{n=1}^\infty X_{A_n}$.

We claim that $Y$ is norm dense in $X$. To this end, it suffices to prove that $Y^+$ is norm dense in $X^+$. Pick a function $f\in X^+$. Then the sequence $(f\chi_{A_n})_{n\in\mathbb N}$ is increasing pointwise to $f$, so that $f_n\to f$ in order. Since $X$ is order continuous, we conclude that $f_n\to f$ in norm which proves the claim. Separability of $X$ follows now from separability of $Y$ and its density in $X$.

(1)$\Leftrightarrow$(5)
If $X$ is norm separable, then (1)$\Rightarrow$(3) implies that the measure $\mu$ is $\sigma$-finite, so that by
\cite[Corollary 5.22]{Abramovich:02} the function space $X$ contains a function $f$ which is strictly positive almost everywhere on $\Omega$. Such a function is a weak unit, so that from order continuity of the norm we conclude that $f$ is a quasi-interior point in $X$. Separability is inherited to $(C_f,\rho)$.

On the other hand, if $X$ admits a function $f$ which is strictly positive almost everywhere on $\Omega$, then $f$ is a quasi-interior point in $X$. To finish the proof we apply \Cref{separability in terms of components}.
\end{proof}

Suppose that $\mu$ is $\sigma$-finite. Then for each ideal $X$ in $L^0(\mu)$ by \cite[Theorem 1.92]{Abramovich:02} there exists the smallest measurable subset $C_X\subseteq\Omega$ with respect to $\mu$-almost everywhere set inclusion  such that every $f\in X$ vanishes $\mu$-almost everywhere on $\Omega\setminus C_X$. The set $C_X$ is called the \term{carrier} (or the \term{support}) of the ideal $X$. By \cite[Lemma 1.94]{Abramovich:02}, the function space $X$ is always order dense in $L^0(C_X)$. This immediately yields that a function space is order dense in $L^0(\mu)$ if and only if its carrier is $\Omega$.  If $X$ is a separable Banach function space, then $X$ has a carrier $C_X$. Indeed, if $X$ is separable, then $X$ contains a quasi-interior point $\varphi$.  
Let $A:=\{x\in \Omega:\; \varphi(x)>0\}$. Since for each $g\geq 0$ in $X$ we have 
$g\wedge n\varphi\to g$, the function $g$ is zero almost everywhere on $X\setminus A$ and so $g\in X_A$. This proves that $X=X_A$. Since $\varphi$ is strictly positive on $A$, we have $C_X=A$. For Banach function spaces over semi-finite measures their separability can be characterized as follows. 

\begin{thm}\label{separabilityBFSEXT}
For an order continuous Banach function space over a semi-finite measure the following assertions are equivalent. 
\begin{enumerate}
\item $X$ is separable.
\item $X$ is separable with respect to the topology of local convergence in measure and the carrier $C_X$ is $\sigma$-finite.
\end{enumerate}
\end{thm}

\begin{proof}
(1)$\Rightarrow$(2)
Since $X$ is separable, it has a quasi-interior point $\varphi$. By the discussion preceding the theorem, the set $A:=\{x\in \Omega:\; \varphi(x)>0\}$ is the carrier $C_X$  of $X$. We claim that $X$ is order dense in $L^0(\mu|_A)$. To see this, pick any function $0<f\in L^0(\mu|_{A})$. Since $f$ is non-zero, there exists $n\in \mathbb N$ such that the set $B:=\{x\in A:\; f(x)\geq \frac 1n\}$ has a positive measure. Then $\chi_B\leq nf$. By the definition of a carrier there exists a non-zero function $g\in X^+$ such that  $g\wedge \chi_B\neq 0$. Hence, $0\leq \frac{1}{n}(g\wedge \chi_B)\leq f$, so that $X_{A}$ is order dense in $L^0(\mu|_{A})$. Since $X=X_{A}$ is norm separable and $\mu|_{A}$ is semi-finite, by \Cref{separabilityBFS} we conclude that the measure $\mu|_A$ is $\sigma$-finite and that $X=X_{A}$ is separable with respect to the topology of local convergence in measure on $A$. We conclude the proof of (2) by observing that  each function in $X$ is zero almost everywhere on $\Omega\setminus A$.  

(2)$\Rightarrow$(1) The same argument as in the proof of the previous implication shows that $X=X_{A}$ is order dense in $L^0(\mu|_A)$, Hence, by \Cref{separabilityBFS} we conclude that $X=X_{A}$ is norm separable. 
\end{proof}

\Cref{separabilityBFS} and \Cref{separabilityBFSEXT} should be compared with \cite[Theorem 46.4 and Theorem 46.5]{Lux14b}.

The one-dimensional Banach function space $L^\infty(\mu)$ over a singleton set with infinite measure shows that neither  \Cref{separabilityBFS} nor \Cref{separabilityBFSEXT} do not hold whenever the underlying measure space is not semi-finite.  

The equivalence between (1) and (3) in \Cref{separabilityBFS} directly shows that over a semi-finite measure $\mu$ the space $L^p(\mu)$ is separable if and only if the $L^q(\mu)$ is separable for $1\leq p,q<\infty$. The following corollary shows that this is also the case for general measure spaces. 

\begin{cor}\label{separabilnost L^p}
For every measure space $(\Omega,\mathcal F,\mu)$ and $1\leq p,q<\infty$ the Banach lattice $L^p(\mu)$ is separable if and only if $L^q(\mu)$ is separable. 
\end{cor}

\begin{proof}
Suppose that $L^p(\mu)$ is separable. Then $L^p(\mu)$ contains a quasi-interior point $f$. As before, let us denote the set $\{x\in \Omega:\; f(x)>0\}$ by $A$. Since $f$ is a quasi-interior point in $L^p(\mu)$, we have $L^p(\mu)=L^p(\mu|_A)$, and since the function $f^{\frac pq}$ is a quasi-interior point in $L^q(\mu)$ we have  $L^q(\mu)=L^q(\mu|_A)$. To finish the proof note that \Cref{separabilityBFS} yields that $L^p(\mu|_A)$ (resp., $L^q(\mu|_A)$) is separable if and only if the topological space $(B,d_{\mu|_B})$ is separable for each subset $B$ of $A$ of finite measure. 
\end{proof}

So far the results in this section provided a connection between separability of a given Banach function space and separability of the underlying measure space. In the following proposition we establish a suitable topological version for $L^1$-spaces.

\begin{prop}\label{Prop5.5}
Let $\Omega$ be a second countable topological space.
If $(\Omega,\mathcal F,\mu)$ is a measure space with an outer regular Borel measure,
then $L^p(\mu)$ is separable for each $1\leq p<\infty$.
\end{prop}

\begin{proof}
By \Cref{separabilnost L^p} it is enough to prove that $L^1(\mu)$ is separable. If there are no sets of positive finite measure in $\Omega$, then  $L^p(\mu)=\{0\}$. Otherwise, let $\mathcal W$ be a countable basis for $X$ and let $\mathcal W'$ be the family of all finite unions of elements of $\mathcal W$ of finite measure.  We will show that the non-empty countable set 
$$
{\mathcal S}=\bigg\{\sum_{W\in K} q_W\chi_W : K\subset{\mathcal W}'\text{ is a finite set and } q_W\in\QQ\bigg\}
$$ 
is dense in $L^1(\mu)$.
Let $f\in L^1(\mu)$ and $\varepsilon>0$. Then there exists a step function $\sum_{k=1}^n q_k\chi_{A_k}\in L^1(\mu)$ such that $\|f-\sum_{k=1}^n q_k\chi_{A_k}\|_1<\tfrac\varepsilon 2$ where $q_k\in\QQ$.
Since $\mu$ is outer regular for every $k$ there exists an open set $V_k\supseteq A_k$ such that $\mu(V_k\setminus A_k)<\tfrac{\varepsilon}{2^{k+2} \,\max\{1,|q_k|\}}$. Since $V_k$ is a (countable) union of elements from $\mathcal W$, there exists $W_k\in{\mathcal W}'$, such that $W_k\subseteq V_k$ and $\mu(V_k\setminus W_k)<\tfrac{\varepsilon}{2^{k+2} \,\max\{1,|q_k|\}}$. Then 
\begin{align*}
    \|\chi_{A_k}-\chi_{W_k}\|_1=\mu(A_k\setminus W_k)+\mu(W_k\setminus A_k)\le \mu(V_k\setminus W_k)+\mu(V_k\setminus A_k)<\tfrac{\varepsilon}{2^{k+1} \,\max\{1,|q_k|\}},
\end{align*}
and therefore 
\[
\bigg{\|}f-\sum_{k=1}^n q_k\chi_{A_k}\bigg{\|}_1\le \bigg{\|}f-\sum_{k=1}^n q_k\chi_{W_k}\bigg{\|}_1+\bigg{\|}\sum_{k=1}^n q_k\chi_{W_k}-\sum_{k=1}^n q_k\chi_{A_k}\bigg{\|}_1<\varepsilon. \qedhere
\]
\end{proof}

The following example shows that there exist a topological space that is not second countable, yet the corresponding $L^1$-space is separable.

\begin{ex}\label{Sorgenfrey}
Consider the S\"orgenfrey line $\RR_S$, i.e., the real line $\RR$ equipped with the topology defined by the basis $\{[a,b) : a,b\in\RR\}$. It is well known that $\RR_S$ is separable but not second countable. The corresponding Borel $\sigma$-algebra $\mathcal B_S$ is the standard Borel $\sigma$-algebra on $\RR$, and hence, $L^1(\RR_S,\mathcal B_S,m)$ is separable.
\end{ex}

In comparison with the Sorgenfrey line one can also find a non-separable topological space $\Omega$ for which the function space $L^1(\mu)$ is still separable.

\begin{ex}
Let $\tau_E$ be the Euclidean topology on $\RR$ and
let 
\[
{\mathcal B}=\{U\setminus C : U\in\tau_E\text{ and }C\text{ is countable}\}
\]
be a basis of a topology $\tau$ on $\RR$. Let $\mathcal B_\tau$ and $\mathcal B_E$ be the Borel $\sigma$-algebras generated by $\tau$ and $\tau_E$, respectively. Because $\tau$ is stronger than $\tau_E$, we have $\mathcal B_E\subseteq \mathcal B_\tau$.
On the other hand, every set $U\setminus C$ is in $\mathcal B_E$, so that ${\mathcal B}\subseteq \mathcal B_E$, and therefore, $\mathcal B_\tau\subseteq \mathcal B_E$. As in the previous example we have that $L^1(\RR,\mathcal B_\tau, m)$ is separable. However, the space $(\RR,\tau)$ is not separable, since for every countable set $Q$ we have $Q\cap (\RR\setminus Q)=\emptyset$ and $\RR\setminus Q\in\tau$.
\end{ex}

\section{The metric space of components}\label{Section: Metric space components}

In this section we are interested in finding conditions on the underlying lattice under which either $I_u=I_v$ or $B_u=B_v$ implies that the corresponding spaces of components $C_u$ and $C_v$ are homeomorphic or even isometric with respect to their natural metrics induced by the norm. First we recall the classical representation theory due to Kakutani.

Given a positive vector $x$ in a vector lattice $X$, consider the principal ideal $I_x$ generated by the vector $x$. It is known that
$$I_x=\{y\in X:\; |y|\leq \lambda x\textrm{ for some }\lambda\geq 0\}$$
and that the mapping $\|\cdot\|_x$ on $I_x$ defined as
$$\|y\|_x:=\inf\{\lambda>0:\; |y|\leq \lambda x\}$$ is a lattice seminorm on $I_x$. If $X$ is Archimedean, then $\|\cdot\|_x$ is a lattice norm on $I_x$ such that for every $y\in I_x$ we have $|y|\leq \|y\|_x\, x$.
If $X$ is uniformly complete, then $(I_x,\|\cdot\|_x)$ is a Banach lattice with a strong unit $x$ so that by an application of the classical Kakutani representation theorem of AM-spaces with strong units (see \cite[Theorem 4.29]{Aliprantis:06}) 
$(I_x,\|\cdot\|_x)$ is lattice isometric to a Banach lattice $C(K)$ for some compact Hausdorff space $K$.  Furthermore, one can require that the isomorphism maps $x$ to the constant function $\one_K$ on $K$.

We start by a simple lemma whose proof is provided only for the sake of completeness.

\begin{lemma}\label{isomorphism of principal ideals}
Let $u$ and $v$ be positive vectors of an Archimedean vector lattice $X$. If $I_u=I_v$, then the identity mapping  $I\colon (I_u,\|\cdot\|_u)\to (I_v,\|\cdot\|_v)$ is an isomorphism of normed lattices.
\end{lemma}

\begin{proof}
Since $I_u=I_v$, the identity mapping between $I\colon (I_u,\|\cdot\|_u)\to (I_v,\|\cdot\|_v)$ is clearly a lattice isomorphism. To prove that it is also an isomorphism of normed lattices, we first find $\lambda\geq 0$ such that $u\leq \lambda v$. Then for $x\in I_u$ the inequality $|x|\leq \|x\|_u u$ immediately yields that $|x|\leq \lambda \|x\|_u v$, from where it follows that $\|x\|_v\leq \lambda \|x\|_u$. This proves that $I\colon (I_u,\|\cdot\|_u)\to (I_v,\|\cdot\|_v)$  is continuous. Similarly one can prove that its inverse is continuous as well.
\end{proof}

Consider the vector lattice $C(K)$ of all real-valued continuous functions ordered pointwise on a non-empty compact Hausdorff space $K$. Pick non-negative functions $u,v\in C(K)$  such that  $I_u=I_v$. 
Then their zero sets $\mathcal Z_u=\{t\in K:\; u(t)=0\}$ and $\mathcal Z_v=\{t\in K:\; v(t)=0\}$ coincide. Since $I_u=I_v$, there exists $\lambda\geq 0$ such that $u\leq \lambda v$ from where it follows that for each $x\in I_u=I_v$ the function $Sx$ defined by  
$$(Sx)(t)=
\begin{cases}
\frac{v(t)}{u(t)}x(t) ,& t\notin \mathcal Z_u,\\
0, & t\in \mathcal Z_u,
\end{cases}$$
is continuous. This yields that $S\colon C_u\to C_v$ is an isomorphism of lattices which  in general is not isometric. Nevertheless, the Kakutani representation theorem allows us to modify $S$ in a way that it becomes also isometric. 

\begin{prop}\label{arch unif compl components}
Let $X$ be a uniformly complete Archimedean vector lattice. Suppose that for positive vectors $u$ and $v$ we have
$I_u=I_v$. Then there exists a surjective isometry between metric spaces $C_u$ and $C_v$.
\end{prop}

\begin{proof}
By \Cref{isomorphism of principal ideals}, $(I_u,\|\cdot\|_u)$ and $(I_v,\|\cdot\|_v)$ are isomorphic as Banach lattices. By  the Kakutani representation theorem there exist compact Hausdorff spaces $K$ and $L$ such that $(I_u,\|\cdot\|_u) $ and $(I_v,\|\cdot\|_v)$ are lattice isometric to $(C(K),\|\cdot\|_\infty)$ and $(C(L),\|\cdot\|_\infty)$, respectively. From this we conclude that $C(K)$ and $C(L)$ are isomorphic as Banach lattices. Let $T\colon C(K) \to C(L)$ be
a Banach lattice isomorphism between $C(K)$ and $C(L)$. By \cite[Theorem 2.34]{Aliprantis:06} there exists a unique non-negative function $g\in C(L)$ and a mapping $\zeta\colon L\to K$ which is continuous on the set $\{y\in L:\; g(y)>0\}$ satisfying
$$(Tf)(y)=g(y)f(\zeta(y)).$$ If $g(y)=0$ would hold for some $y\in L$, then $(Tf)(y)=0$ for all $f\in C(K)$ which would imply that $T$ is not surjective. Hence, $g$ is strictly positive, so that $\zeta$ is continuous on $L$. Since $g$ is strictly positive and since $T$ is a lattice isomorphism, the mapping
$S\colon C(K)\to C(L)$ defined as $(Sf)(y)=f(\zeta(y))$ is a lattice and algebra isomorphism which satisfies $S(\one_K)=\one_L$. An application of \cite[Lemma 4.14]{Eisner} yields that the mapping $\zeta$ is a homeomorphism, so that the mapping $S$ is isometric.

Hence, there exist a Banach lattice isometric isomorphism $\widetilde S\colon (I_u,\|\cdot\|_u) \to (I_v,\|\cdot\|_v)$ which maps $u$ to $v$.
To conclude the proof observe that its restriction $\widetilde S|_{C_u}$ is a surjective isometry between metric spaces $C_u$ and $C_v$.
\end{proof} 

\begin{quest}
Does \Cref{arch unif compl components} still hold if the assumption of uniform completeness is removed? 
\end{quest}

Since Banach lattices are Archimedean and uniformly complete, the following corollary immediately follows \Cref{arch unif compl components}.

\begin{cor}\label{Korolar6.3}
Let $u$ and $v$ be positive elements of a Banach lattice $X$. If $I_u=I_v$, then there exists a surjective isometry between metric spaces $C_u$ and $C_v$.
\end{cor}

The natural question whether the equality between principal ideals could be replaced by the equality between principal bands has unfortunately a negative answer even when the underlying Banach lattice is Dedekind complete as the following example shows.  

\begin{ex}\label{discontinuity of varphi}
Consider the Banach lattice $\ell^\infty$ of all bounded sequences. Since the vector $\one:=(1,1,\ldots)$ is a strong unit in $\ell^\infty$ and the vector $u=(1,\frac 12,\frac 13,\ldots)$ is a weak unit in $\ell^\infty$ we clearly have that $\ell^\infty=B_\one=B_u$. Since we  have
$$
C_u=\{u\chi_A:\;A\subseteq \mathbb N\}\quad\text{and}\quad C_\one=\{\chi_A:\;A\subseteq \mathbb N\}, 
$$
it follows that the mapping
$u\chi_A \mapsto \chi_A$ is a bijection between $C_u$ and $C_\one$. Since the metric space $C_\one$ is discrete, any mapping $\phi\colon C_\one \to C_u$ is continuous.

We claim that there is no continuous injective mapping $\phi\colon C_u\to C_\one$. 
To this end, suppose that there exists a continuous injective mapping $\phi\colon C_u\to C_\one$.
For each $n\in \mathbb N$ let us consider the component $u_n:=(1,\ldots,\frac 1n,0,\ldots)$ of $u$.
Since $u_n\to u$ in $\ell^\infty$, we would have $\phi(u_n)\to \phi(u)$. Since both $\phi(u_n)$ and $\phi(u)$ are distinct components of $\one$ in $\ell^\infty$, the sequence $\phi(u_n)-\phi(u)$ would need to have at least one non-zero term. Since the components of the vector $\one$ in $\ell^\infty$ are sequences consisting only of zeros and ones, we would conclude that for each $n\in \mathbb N$ we would have $\|\phi(u_n)-\phi(u)\|\geq 1$. This contradiction shows that $C_u$ and $C_\one$ are not homeomorphic.
\end{ex}

It turns out that in \Cref{Korolar6.3} we can replace principal ideals by principal bands in the case when the underlying Banach lattice is order continuous (see \Cref{homeomorphism between space of components}). We will prove this result by  an explicit construction of a homeomorphism between $C_u$ and $C_v$. The definition of the aforementioned homeomorphism will rely on a very easy yet useful description of components of a given vector in the case when the lattice has the principal projection property.

\begin{lemma}\label{opis komponent}
Let $u$ be a positive vector of a vector lattice $X$ with the principal projection property. Then
$$C_u=\{Pu:\; P \textrm{ is a principal band projection }\}.$$
\end{lemma}

\begin{proof}
Let us denote the right-hand side set appearing in the lemma by $S$. 
If $P$ is a principal band projection, then for every positive vector $u\in X$ we have $Pu\perp (u-Pu)$ and $u=Pu+(u-Pu)$. Hence,  $Pu$ is a component of $u$ which proves that $S\subseteq C_u$.

To prove that $C_u\subseteq S$, pick $x\in C_u$ and consider the principal band projection $P$ onto the principal band generated by $x$.
Since $x\perp (u-x)$, we have that $P(u-x)=0$, and since $Px=x$ we conclude that $Pu=P(u-x)+Px=x$.
\end{proof}

Now it should be quite obvious that the required homeomorphism $\varphi\colon C_u \to C_v$  should be defined as $\varphi\colon Pu \mapsto Pv$ for all principal band projections $P$.

\begin{lemma}\label{varphi bijection}
Let $X$ be a vector lattice with the principal projection property. Suppose that
$B_u=B_v$ for some positive vectors $u,v\in X$.
\begin{enumerate}
\item Then the mapping $\varphi\colon C_u\to C_v$ is a well-defined bijection.
\item If  $X$ is a normed lattice and $u\in I_v$, then $\varphi\colon C_v\to C_u$ is a continuous bijection.
\end{enumerate}
\end{lemma}

\begin{proof}
It suffices to prove the statement only for the special case when $u$ and $v$ are weak units of $X$. Indeed, suppose that the statement holds for the case of weak units and consider the principal band $X_0:=B_u=B_v$ of $X$. By definition, $u$ and $v$ are weak units of $X_0$, and since $X_0$ is also an ideal in $X$, by \cite[Theorem 25.2]{Luxemburg:71} the vector lattice $X_0$ has the principal projection property. Since $C_u$ and $C_v$ are subsets of $X_0$, the general result follows from the special case.

(1) We need to check first that $\varphi$ is well-defined. Since by \Cref{opis komponent} $Pv\in C_v$ for each principal band projection $P$, it remains to prove that
$Pu=Qu$ yields $Pv=Qv$. Assume, therefore, that $Pu=Qu$ for some principal band projections $P$ and $Q$.
Since $Pu=Qu$ trivially implies $(P-Q)^2u=0$ and since
$$(P-Q)^2=P(I-Q)+Q(I-P)\geq 0$$ is an order continuous operator, the fact that $u$ is a weak unit in $X$ yields that $(P-Q)^2=0$. On the other hand, from positivity of operators $P(I-Q)$ and $Q(I-P)$ and from the identity $P(I-Q)+Q(I-P)=(P-Q)^2=0$ we conclude that $P=PQ$ and $Q=QP$. Since band projections always commute, we finally obtain that $P=Q$ which proves that $\varphi$ is well-defined.

The same argument shows that the mapping $\psi\colon C_v\to C_u$ defined by $\psi(Pv)=Pu$ is well-defined which is clearly the inverse of $\varphi$. 

(2) Suppose that a sequence $(P_nv)_{n\in\mathbb N}$ converges to some component $Pv$ in $C_v$. Since $u\in I_v$, there exists $\lambda\geq 0$ such that $u\leq \lambda v$. Hence, from the inequality
$$|P_nu-Pu|\leq \lambda |P_nv-Pv|$$ we conclude that
$P_nu\to Pu$, so that $\varphi$ is continuous.
\end{proof}

The following corollary should be compared with \Cref{arch unif compl components} where an existence of a surjective isometry between $C_u$ and $C_v$ is proved for uniformly complete Archimedean vector lattices. If we replace the assumption on uniform completeness with the principal projection property,  the mapping $\varphi$ defined above is only a homeomorphism. Moreover, $\varphi$ is isometric if and only if $u=v$. This will follow from \Cref{varphi isometrija}.

\begin{cor} 
Let $X$ be a vector lattice with the principal projection property. If $I_u=I_v$ for some positive vectors $u$ and $v$, then the mapping $\varphi\colon C_u\to C_v$ is a homeomorphism.
\end{cor}

If we replace the condition $I_u=I_v$ in the corollary above by a weaker condition $B_u=B_v$, in general, the mapping $\varphi$ may not be continuous even when $X$ is Dedekind complete as it was shown in \Cref{discontinuity of varphi}. Continuity of $\varphi$ follows automatically when the underlying Banach lattice is order continuous. Before we state and prove \Cref{homeomorphism between space of components} which is the main result of this section, we need the following result which is of independent interest.

\begin{prop}\label{konvergenca komponent}
Let $X$ be an Dedekind complete normed lattice which has the $\sigma$-Fatou property or is a Banach lattice.  Suppose that $(P_n)_{n\in\mathbb N}$ is a sequence of band projections such that $P_nu\to Pu$ for some band projection $P$ and some $u\in X^+$. Then there exists a subsequence $(P_{n_k})_{k\in\mathbb N}$ which converges in order to $P$ on $B_u$ and strongly on $\overline{I_u}$ . 
\end{prop}

\begin{proof}
Let us first consider the case when $P=0$.
Then $P_nu\to 0$, so that there is a subsequence $(P_{n_k})_{k\in\mathbb N}$ such that for each $k\in \mathbb N$ we have
$\|P_{n_k}u\|\leq \frac{1}{k^2}$. For each $k\in \mathbb N$ we consider the band projection
$$Q_k:=\bigvee_{j=k}^\infty P_{n_j}.$$ It is clear that the sequence $(Q_k)_{k\in\mathbb N}$ is decreasing. 

We claim that $Q_ku\to 0$. Consider first the case when $X$ is a Banach lattice. Then for each $k$ the series $\sum_{j=k}^\infty P_{n_j}u$ converges in $X$ to some positive vector $u_k$ which satisfies
$$\|u_k\|\leq \sum_{j=k}^\infty \|P_{n_j}u\|\leq  \sum_{j=k}^\infty \frac{1}{j^2}\to 0$$ as $k\to\infty$.
Since for each $l\geq k$ we have
$\big(\bigvee_{j=k}^lP_{n_j}\big)u \leq \sum_{j=k}^lP_{n_j}u\leq u_k$, due to order closedness of the positive cone we conclude 
$$\bigg(\bigvee_{j=k}^\infty P_{n_j}\bigg)u \leq u_k.$$ 
Since $u_k\to 0$, we have  $Q_ku=\big(\bigvee_{j=k}^\infty P_{n_j}\big)u\to 0$ as $k\to\infty$.

Let us now consider the case that $X$ has the $\sigma$-Fatou property. Pick any  $\epsilon>0$. Since $\sum_{j=k}^\infty \|P_{n_j}u\|\leq \sum_{j=k}^\infty \frac{1}{k^2}\to 0$ as $k\to\infty$, there exists $k_0\in \mathbb N$ such that
$$\bigg\|\bigg(\bigvee_{j=k}^lP_{n_j}\bigg)u\bigg\| \leq \sum_{j=k}^l \|P_{n_j}u\| \leq \sum_{j=k}^\infty \|P_{n_j}u\|<\epsilon$$ for all $l\geq k\geq k_0$. Since $\big(\bigvee_{j=k}^lP_{n_j}\big)u \uparrow Q_ku$ as $l\to \infty$ and since $X$ has the $\sigma$-Fatou property, for each $k\geq k_0$ we obtain $\|Q_ku\|\leq \epsilon$. This proves that $Q_ku\to 0$ also in the second case.

To prove that $P_{n_k}\to 0$ strongly on $\overline{I_u}$, pick a positive vector $x\in \overline{I_u}$ and an $\epsilon>0$. Since $u$ is a quasi-interior in $\overline{I_u}$, one can find $n\in \mathbb N$ such that $\|x-x\wedge nu\|<\frac{\epsilon}{2}$. Since $Q_ku\to 0$, there exists $k_0\in \mathbb N$ such that for all $k\geq k_0$ we have $\|Q_ku\|<\frac{\epsilon}{2n}$.
Since for each $k\geq k_0$ we have
\begin{align*}
\|P_{n_k}x\|\leq \|Q_kx\|\leq \|Q_k(x-x\wedge nu)\|+\|Q_k(x\wedge nu)\|
\leq \tfrac{\epsilon}{2} + \|Q_k(nu)\|<\epsilon,
\end{align*}
we have $P_{n_k}x\to 0$.

Now we prove that $P_{n_k}\to 0$ in order on $B_u$. Since $0\leq P_{n_k}\leq Q_k$ on $B_u$, it suffices to prove that $Q_k\downarrow 0$ on $B_u$. As $B_u$ is Dedekind complete, by 
\cite[Theorem 1.19]{Aliprantis:06} it is enough to prove that 
$Q_kx\downarrow 0$ for each $x\in B_u^+$. Suppose that some vector $y\in B_u^+$ satisfies $0\leq y\leq Q_kx$ for each $k\in \mathbb N$ and observe that for a fixed positive integer $n$  and for all positive integers $k\geq k_0$ we have 
$$
0\leq y\leq Q_kx=Q_k(x-x\wedge nu)+Q_k(x\wedge nu)\leq Q_{k_0}(x-x\wedge nu)+nQ_ku.$$
Since $Q_ku\downarrow$ and $Q_ku\to 0$, we have that $Q_ku\downarrow 0$ and so by letting $k$ go to infinity, for each $n\in\mathbb N$ we obtain 
$$0\leq y\leq Q_{k_0}(x-x\wedge nu).$$ Since $u$ is a weak unit in $B_u$, we have $x\wedge nu\uparrow x$ as $n\to\infty$, so that order continuity of $Q_{k_0}$ yields $Q_{k_0}(x-x\wedge nu)\downarrow 0$. This proves that $y=0$ from where it follows  $Q_kx\downarrow 0$.

For the proof of the general case, let us denote by $R_n$ the operator $P_n-P$. Since
$R_n=P_n-P_nP+P_nP-P=P_n(I-P)-P(I-P_n)$ and since band projections $P_n(I-P)=(I-P)P_n$ and $P(I-P_n)$ have disjoint ranges, they are disjoint in the vector lattice of all regular operators. This yields that
$|R_n|=P_n(I-P)+P(I-P_n)$ which implies that $|R_n|$ is a band projection. Since $R_nu\to 0$, by the special case above one can find a subsequence $(R_{n_k})_{k\in\mathbb N}$ which converges to $0$  strongly on $\overline{I_u}$ and in order on $B_u$.
\end{proof}

\Cref{konvergenca komponent} can be applied in two particular cases. If $u$ is a weak unit in $X$, then there is a subsequence $(P_{n_k})_{k\in\mathbb N}$ which converges in order to $P$ on $B_u=X$, and if $u$ is a quasi-interior point in $X$, then there is a subsequence $(P_{n_k})_{k\in\mathbb N}$ which converges in order and strongly to $P$ on $\overline{I_u}=X$.

\begin{thm}\label{homeomorphism between space of components}
Let $X$ be an order continuous Banach lattice. If two positive vectors $u$ and $v$ satisfy $B_u=B_v$ then $\varphi\colon C_u\to C_v$ is a homeomorphism of metric spaces.
\end{thm}

\begin{proof}
We first consider the case when $u$ and $v$ are weak units of $X$.
It follows from \Cref{varphi bijection} that $\varphi\colon C_u\to C_v$ defined as $\varphi(Pu)= Pv$ is a bijection where $P$ runs over all principal band projections.

We claim that $\varphi$ is also continuous.
Suppose that a sequence $(x_n)_{n\in\mathbb N}$ of vectors in $C_u$ converges to a vector $x_0\in C_u$. For each $n\in \mathbb N_0$ there exists a principal band projection $P_n$ such that $x_n=P_nu$.
Suppose that the sequence $(P_nv)_{n\in\mathbb N}$ does not converge to zero. Hence, by passing to a subsequence (if necessary) there exists $\epsilon>0$ such that $\|P_{n}v-Pv\|\geq \epsilon$ for all $n\in \mathbb N$.
By \Cref{konvergenca komponent} there exists a subsequence $(P_{n_k})_{k\in \mathbb N}$ which converges in order to $P$. Again, by passing to a further subsequence (if necessary), we may also assume that $(P_n)_{n\in\mathbb N}$ converges to $P$ in order. Hence, there exists a sequence of positive operators $(A_n)_{n\in\mathbb N}$ such that $0\leq |P_n-P|\leq A_n\downarrow 0$.
Since $X$ is order continuous, from the  inequality
$$0\leq |P_nv-Pv|=|P_n-P|v\leq A_nv\to 0$$ we conclude that $P_nv\to Pv$ in norm which contradicts the fact that $\|P_nv-Pv\|\geq \epsilon$ for all $n\in \mathbb N$. Similarly, one can prove that the inverse mapping $\varphi^{-1}\colon Pv\to Pu$ is also continuous.

For the proof of the general case, consider the closed sublattice $F:=B_u=B_v$ and note that $F$ in its own right is an order continuous Banach lattice with quasi-interior points $u$ and $v$.
\end{proof}

The following proposition shows that the mapping $\varphi\colon C_u\to C_v$ is never an isometry unless $u=v$.

\begin{prop}\label{varphi isometrija}
Let $X$ be a normed lattice with the principal projection property. Suppose that for positive vectors $u$ and $v$ we have $B_u=B_v$. If the natural mapping
$\varphi\colon C_u\to C_v$ is an isometry, then $u=v$.
\end{prop}

\begin{proof}
As in the proof of \Cref{varphi bijection} and \Cref{homeomorphism between space of components} it is enough to prove the statement for the case when $u$ (and so also $v$) is a weak unit. 

Pick any $\alpha>1$ and consider the vector $w:=(u-\alpha v)^+$. Let $P_{w}$ be the band projection onto the principal band generated by the vector $w$.
Since
$$P_w(u-\alpha v)=P_w(u-\alpha v)^+-P_w(u-\alpha v)^-=P_w(u-\alpha v)^+\geq 0,$$ we conclude that
$P_wu\geq \alpha P_w(v)$. Since $\varphi$ is an isometry, we have 
$$\|P_wv\|=\|\varphi(P_wu)\|=\|P_wu\|\geq \alpha \|P_wv\|,$$ so that $P_wv=0$. Since $v$ is a weak unit and since $P_w$ is positive and order continuous, it follows that $P_w=0$.
This yields that $(u-\alpha v)^+=0$ from where it follows that $u\leq \alpha v$ for each $\alpha>1$. Letting $\alpha\downarrow 1$ we obtain $u\leq v$. Similarly one can show that $v\leq u$.
\end{proof}

\section{Some remarks on the normality of the unbounded topology}\label{Section: normal un}

We conclude this paper with this short section where we consider the question under which conditions a given normed lattice $X$ equipped with its unbounded norm topology is a normal topological space. Since metrizable spaces are normal, $(X,\taun)$ is certainly normal in this case. By \cite[Theorem 4.3]{KT18}, $\taun$ is metrizable if and only if $X$ contains an at most countable set  which generates a norm dense ideal in $X$. In particular, if $X$ is a Banach lattice with a quasi-interior point, then $(X,\taun)$ is metrizable. 

The following theorem whose proof is a combination of results from general topology, topological vector spaces and the unbounded norm topology provides a more general situation when $(X,\taun)$ is a normal space.
In order to do that we first recall the following notions. A Banach lattice $X$ is a \term{KB-space} whenever every increasing norm bounded sequence in $X^+$ is norm convergent.  
A Hausdorff topological space is \term{paracompact} whenever it admits a partition of unity subordinate to any open cover.

\begin{thm}\label{paracompact}
If $X$ is an atomic KB-space, then $(X,\taun)$ is a paracompact topological space.
\end{thm}

\begin{proof}
Since $X$ is an atomic KB-space, the closed unit ball $B_X$ is un-compact by \cite[Theorem 7.5]{KMT}, so that $(X,\taun)$ is a $\sigma$-compact topological space. By \cite{DOT}, the unbounded norm topology on $X$ is a locally solid Hausdorff topology, so that $(X,\taun)$ is regular (see e.g. \cite{SCH71}). Since $(X,\taun)$ is $\sigma$-compact, it is clearly a Lindel\"of space, so that paracompactness of the unbounded norm topology finally follows from Morita's theorem \cite[Theorem VIII.6.5]{Dugundji:66}.
\end{proof}

The following corollary is an easy consequence of \cite[Theorem VIII.2.2]{Dugundji:66} and \Cref{paracompact}.

\begin{cor}
If $X$ is an atomic KB-space, then $(X,\taun)$ is a normal topological space.
\end{cor}

The proof of \Cref{paracompact} heavily depends on Morita's theorem from general topology. It would be of interest to find a functional analytical proof. We conclude the paper with the following open question.

\begin{prob}
Under which conditions on a Banach lattice $X$ the topological space $(X,\taun)$ is normal?
\end{prob}

The answer remains  unknown even for order continuous atomic Banach lattices.

\bibliographystyle{alpha}
\bibliography{separability}

\end{document}